\documentclass[12pt]{amsart}
\usepackage{amsfonts}

\usepackage[letterpaper]{geometry}

\usepackage{amssymb}
\def\C{\mathbb{C}}

\def\F{{\mathbb F^*_p}}
\def\f{{\mathbb F_p}}

\newtheorem{thm}{\bf Theorem}[section]
\newtheorem{lemma}[thm]{\bf Lemma}

\newtheorem{cor}[thm]{\bf Corollary}

\theoremstyle{definition}

\begin{document}

\title{On the distribution of multiplicative  energy}

 \author[Norbert Hegyv\'ari]{Norbert Hegyv\'ari}
 \address{Norbert Hegyv\'{a}ri, ELTE TTK,
E\"otv\"os University, Institute of Mathematics, H-1117
P\'{a}zm\'{a}ny st. 1/c, Budapest, Hungary and associated member of Alfr\'ed R\'enyi Institute of Mathematics, Hungarian Academy of Science, H-1364 Budapest, P.O.Box 127.}
 \email{hegyvari@renyi.hu}

\begin{abstract}
This paper extends the investigation of energy distribution in finite settings, building on results established in our earlier work. We analyze the distribution of multiplicative energies of different orders using Fourier-analytic methods. Our results provide new structural insights into energy phenomena in finite fields, complementing earlier discrete analyses. In the last section, as a corollary, we obtain a 'sum-multiplicative energy'-type result.

MSC 2020 Primary 11B75, Secondary 11B30

\end{abstract}

 \maketitle

\section{Introduction}

Additive and multiplicative energies are central notion in additive combinatorics, popularized by T.Tao, and nowadays many authors investigated this topic; we mention here in-depth examinations of energies by Shkredov in [S1], [S2], [S3], as well as related work by others in [BCK], [KT], [TV]. Some nice and elegant discrete-to-continuum bridging results can be found in [PT1], [PT2], [PT3].

In [H] the author investigated the range and distribution of additive energy values and  related inverse problems in a general finite setting for instance the distribution and "gaps" between possible values of additive energy for sets within the $n$-dimensional grid $[M]^n$. Specifically, it proves the existence of a long sequence of energy values where the ratio of consecutive elements $E(X_{i+1})/E(X_i)$ is bounded near 1. This suggests that in certain regimes the spectrum of possible additive energies is quite dense. Furthermore, it is proved that the sequence of energies cannot be too dense; there is an effectively computable constant $C$ (depending on $M$ and $n$) such that the ratio of any two consecutive energy values $E_{j+1}/E_j$ must remain above a certain threshold.

In this paper, we extend these investigations to the framework of finite fields, where the interaction between additive and multiplicative structures becomes particularly significant. Using Fourier-analytic methods, we study relations between multiplicative energies of distinct orders. This approach provides new structural insights into energy phenomena in finite fields, complementing the author's earlier work [H].

\medskip

{\bf Acknowledgment:} The author wish to thank the anonymous reviewer for his/her careful reading of the manuscript, and for pointing out the inaccuracies.

{\bf AI use:} All the main ideas and mathematical results in this paper were developed by the authors. GPT-5.6 was used to help improve the English and readability of an earlier draft.

\section{Notations}

For a positive integer $N$ let $[N]:=\{1,2,\dots, N\}$. We use Vinogradov's notation, $f \ll g$, to denote that there exists an absolute constant $C$ such that $\vert{}f\vert{} \leq C\vert{}g\vert{}$. We write $g \gg f$ to mean $f \ll g$, and $f \asymp g$ if both $f \ll g$ and $g \gg f$ hold. $\f$ denotes the prime field with characteristic $p$, where $p$ is a prime number, $\F:=\f\setminus \{0\}$.

Let $e_p(x)=e^{2\pi xi/p}$ and for functions $f,g: \f\rightarrow \C.$
let $\widehat{f}(r)=\sum_{x\in \f}f(x)\cdot e_p(rx).$ Recall the classical Plancherel formula $\sum_{x\f}f(x)\overline{g(x)}=\frac{1}{p}\sum_{r\in \f}\widehat{f}(r)\overline{\widehat{g}(r)}$.
Denote by $A(x)$ the indicator function (i.e. $A(x)=1$ when $x\in A$ and $0$ otherwise.) In this special case $\sum_r|\widehat{A}(r)|^2=p\sum_x|A(x)|^2=p\cdot |A|.$
Let $A\subseteq \f$. 

The $k-$fold multiplicative energy of the set $A$ is 
$
E^\times_k(A)=|\{(a_1, a_2,\dots ,a_k,a'_1, a'_2,\dots ,a'_k): a_1\cdot a_2\cdots a_k=a'_1\cdot a'_2\cdots a'_k; \ a_i,a'_i \in A\}|
$
The additive energy is defined in a similar way.

In the next two paragraph we will show how Elekes-Garaev equation leads connection to multiplicative energies.

\section{On multiplicative energy for a set of $\F$}

Let us start with the generalized Elekes–Garaev equation 
\begin{equation}\label{(1)}
   x_1x_2\cdots x_ky_1^{-1}y_2^{-1}\cdots y_{k-1}^{-1} + a = b  
\end{equation}
(where $x_i, y_j, a \in A$ for $i=1, \dots, k$ and $j=1, \dots, k-1$, and $b\in A+A$)
\begin{thm}
Let $A\subseteq \F$. For every $k\geq 2$, we have
 \begin{equation}\label{2}
E^\times_k(A)\leq   \frac{1}{p}|A|^{2k-1}|A+A|+\sqrt{p|A+A|E^\times_{2k-2}(A)}.
 \end{equation}   
\end{thm}

\begin{proof}[Proof of the theorem]

Denote by $S$ the number of solutions of (\ref{(1)}).

When $x_1x_2\cdots x_ky_1^{-1}y_2^{-1}\cdots y_{k-1}^{-1}\in A$ this equation holds. Rearranging equation  $x_1x_2\cdots x_ky_1^{-1}y_2^{-1}\cdots y_{k-1}^{-1}=a'$ one can have a $2k-$tuple for which $x_1x_2\cdots x_k=y_1y_2\cdots y_{k-1}\cdot a'$. Hence, the number of solutions $S$ is at least $E^\times_k(A)|A|$. On the other hand
$$
S=\frac{1}{p}\sum_{r\in \f}\sum_{x_1,\dots,x_k,y_1,\dots, y_{k-1}}\sum_{a,b}e_p(r(x_1\cdots x_ky_1^{-1}y_2^{-1}\cdots y_{k-1}^{-1} + a-b))=
$$
$$
=\frac{1}{p}|A|^{2k}|A+A|+\frac{1}{p}\sum_{r\in \F}\sum_{x_1,\dots,x_k,y_1,\dots, y_{k-1}}\sum_{a,b}e_p(r(x_1\cdots x_ky_1^{-1}y_2^{-1}\cdots y_{k-1}^{-1} + a-b)).
$$
We give an estimation for the second sum $S'$ as:
$$
S'\leq \frac{1}{p}\max_{r\neq 0}\Big|\sum_{x_1,\dots,x_k,y_1,\dots, y_{k-1}}e_p(r(x_1\cdots x_ky_1^{-1}y_2^{-1}\cdots y_{k-1}^{-1} + a-b))\Big|\sum_r\Big|\sum_{a,b}e_p(r(a-b))\Big|:=
$$
$$
=\frac{1}{p}\max_{r\neq 0}T_r\sum_r\Big|\sum_{a,b}e_p(r(a-b))\Big|,
$$
where $T_r:= \Big|\sum_{x_1,\dots,x_k,y_1,\dots, y_{k-1}}e_p(r(x_1\cdots x_ky_1^{-1}y_2^{-1}\cdots y_{k-1}^{-1}))\Big|$
\begin{lemma}\label{3.1}
    For every $r\in \F$ we have $T_r\leq \sqrt{p|A|E^\times_{2k-2}(A)}.$
\end{lemma}
\begin{proof}. By the Cauchy inequality
$$ 
T_r^2\leq |A|\sum_{x_1\in \f}\Big|\sum_{x_2,\dots x_k,y_1,\dots, y_{k-1}}e_p(rx_1(x_2\cdots x_ky_1^{-1}\dots y_{k-1}^{-1}))\Big|^2=
$$
$$=|A|\sum_{x_1\in \f}\sum_{x_2,\dots x_k,y_1,\dots, y_{k-1}}e_p(rx_1((x_2\cdots x_ky_1^{-1}\dots y_{k-1}^{-1})-(x_2'\cdots x'_ky_1'^{-1}\dots y_{k-1}'^{-1}))).
$$
The argument $((x_2\cdots x_ky_1^{-1}\dots y_{k-1}^{-1})-(x_2'\cdots x'_ky_1'^{-1}\dots y_{k-1}'^{-1}))=0$ iff $$((x_2\cdots x_ky_1^{-1}\dots y_{k-1}^{-1})=(x_2'\cdots x'_ky_1'^{-1}\dots y_{k-1}'^{-1})).$$ Counting all $2k-2$-tuples we have $T_r\leq \sqrt{p|A|E^\times_{2k-2}(A)}.$
\end{proof}

The  Cauchy inequality and the Parseval identity gives $\sum_r\Big|\sum_{a,b}e_p(r(a-b))\Big|\leq p\sqrt{|A||A+A|}$.
Hence 
$$
S'\leq \frac{1}{p}p\sqrt{|A||A+A|}\sqrt{p|A|E^\times_{2k-2}(A)}=|A|\sqrt{p|A+A|E^\times_{2k-2}(A)}.
$$
So we have
$$
E^\times_k(A)|A|\leq S\leq   \frac{1}{p}|A|^{2k}|A+A|+|A|\sqrt{p|A+A|E^\times_{2k-2}(A)}
$$
$$
E^\times_k(A)\leq   \frac{1}{p}|A|^{2k-1}|A+A|+\sqrt{p|A+A|E^\times_{2k-2}(A)}.
$$
\end{proof}
Now we derive some consequence of the theorem:
\begin{cor}
 Let $A\subseteq \F$, $|A|>p^{2/3}/(2K)^{1/3}$ and assume that $|A+A|\leq K|A|$, $K>1$. For $k\geq2$
$$
E_k^\times(A)\leq \frac{3K}{p}|A|^{2k},
$$
\end{cor}
By the Cauchy inequality it is easy to see that $E_k^\times(A)\geq\frac{|A|^{2k}}{p-1}$. Hence in the above case the multiplicative energy differs from the smallest possible value by no more than a factor of $3K$. It would be interesting to determine how strict the factor $3K$.
\begin{proof}[Proof of the Corollary].
We start by a simple fact. Let $\chi$ be a multiplicative character of $\F$. By the orthogonality $(p-1)E_k^\times(A)=\sum_\chi|\sum_{a\in A}\chi(a)|^{2k}.$ Hence
$$
(p-1)E_{k+1}^\times(A)=\sum_\chi|\sum_{a\in A}\chi(a)|^{2k+2}\leq |A|^2 \sum_\chi|\sum_{a\in A}\chi(a)|^{2k}=(p-1)|A|^2E_k^\times(A),
$$
hence  $E_{k+1}^\times(A)\leq |A|^2E_k^\times(A).$ This gives
$$
E_{2k-2}^\times(A)\leq |A|^{2k-4}E_k^\times(A).
$$
By (\ref{2}) and by the arithmetic-geometric inequality and by $|A+A|\leq K|A|$
$$
E_k^\times(A)
\leq \frac{|A|^{2k-1}|A+A|}{p}
+\sqrt{p|A+A||A|^{2k-4}E_k^\times(A)}\leq 
$$
$$
\leq \frac{|A|^{2k-1}|A+A|}{p}+\frac{p|A+A||A|^{2k-4}+E_k^\times(A)}{2}
$$
which yields
$$
E_k^\times(A)
\leq
2K\frac{|A|^{2k}}{p}
+Kp |A|^{2k-3}\leq 3K\frac{|A|^{2k}}{p}.
$$
provided $|A|>p^{2/3}/(2K)^{1/3}$
\end{proof}

\section{Concluding remarks}

{\bf On multiplicative energies for four sets}.
We use the Elekes-Garaev equation in a more general setting. Consider the equation
\begin{equation}\label{(3)}
   a_1a_2\cdots a_kb_1^{-1}b_2^{-1}\cdots b_{k-1}^{-1} + c_1c_2\cdots c_md_1^{-1}d_2^{-1}\cdots d_{m-1}^{-1} = y  
\end{equation}
where $a_i\in A, b_j\in B, c_t \in C, d_s\in D$ for $i=1, \dots, k$ and $j=1, \dots, k-1$, $t=1, \dots, m$,  $s=1, \dots, m-1$ and $y\in B+D$.

In a similar argument, one can have a theorem which is analogous to Theorem \ref{3.1}. 

We consider just the case where the equation is of the form
\begin{equation}\label{4}
\frac{x_1x_2}{x_3}+\frac{y_1y_2}{y_3}=z,
\end{equation}
where $x_1,x_2,x_3\in X$, $y_1,y_2,y_3\in Y$, and $z\in X+Y$.

By adapting the argument from Theorem \ref{3.1} and relating it to multiplicative energies, we see that (\ref{4}) has at least $E_2^\times(X)E_2^\times(Y)$ solutions. Using bounds on exponential sums, we establish the following result:
\begin{thm}\label{4.1} Let $X,Y\subseteq \F$ be non empty sets.
We have
$$
E_2^\times(X)E_2^\times(Y)
\leq
\frac{|X|^3|Y|^3}{p}|X+Y|+H(X,Y)
$$
where $H(X,Y)=\min\left\{
|Y|^{5/2}\sqrt{p|X||X+Y|E_2^\times(X)},
|X|^{5/2}\sqrt{p|Y||X+Y|E_2^\times(Y)}
\right\}$
\end{thm}
In the spirit of sum-product phenomena, from this theorem we obtain an immediate sum-energy-type corollary:
\medskip
\begin{cor}
$$
|X+Y|\ge
\min\left\{
\frac{pE_2^\times(X)E_2^\times(Y)}
{2|X|^3|Y|^3},
\frac{(E_2^\times(X)E_2^\times(Y))^{3/2}}
{4p|X|^3|Y|^3}
\right\}.
$$
In particular, when $|X|\asymp |Y|\gg p^{2/3}$, $E_2^\times(X) \asymp E_2^\times(Y) \asymp |X|^3$
$$
|X+Y|\gg p,
$$
i.e., $X+Y$ has positive density in $\f$.
\end{cor}

\vskip 1cm

{\bf Funding} The research is supported by the National Research, Development and Innovation Office NKFIH Grant No K-146387

{\bf Conflict of interest} There are neither conflicts of interest nor competing interests. This manuscript has no associated data.


\end{document}